\documentclass[12pt]{article}
\usepackage{amssymb}
\usepackage{color}
\usepackage[dvipsnames]{xcolor}
\usepackage{amsmath}
\usepackage{cite}
\usepackage{graphicx}
\usepackage{enumerate}
\usepackage{hyperref}
\newenvironment{proof}[1][Proof]{\noindent\textit{#1.} }{\hfill  \rule{0.5em}{0.5em}}

\newtheorem{theorem}{Theorem}
\newtheorem{lm}{Lemma}

\newtheorem{exmp}{Example}

\begin{document}

\title{
A note on 
Fredholm integral equation
}

\author{
Chaitanya Gopalakrishna
\vspace{3mm}\\
{\small Statistics and Mathematics Unit, Indian Statistical Institute,}
\\
{\small  R.V. College Post, Bangalore-560059, India}
    \vspace{0.2cm}\\
    {\small cberbalaje@gmail.com 
    }}

\date{}


\maketitle

\begin{abstract}
This note gives results on the existence of semi-continuous solutions of a Fredholm integral equation of the second kind using Tarski's fixed point theorem.

\vskip 0.2cm

\vskip 0.2cm

\end{abstract}



\baselineskip 16pt
\parskip 10pt


\section{Introduction}

An {\it integral equation} is equality where an unknown function appears under an integral sign. Such equations, which  frequently arise in many scientific and engineering problems,
 have been intensively investigated, and many different powerful methods for obtaining their exact and approximate solutions have been proposed (see   \cite{Agarwal,Atkinson,Colton,Gripenberg,Hsiao,pruss,Wazwaz} for example). In particular, special attention is given to
the study of a Fredholm integral equation of the second kind
\begin{eqnarray}\label{Fredholm}
\phi(t)=f(t)+\lambda\int_{a}^{b}K(t,s)\phi(s)ds,\quad t\in [a,b],
\end{eqnarray} 
 one of the most fundamental integral equations, for deeper analysis.
 A small sample of the literature on  advances in various aspects of it and its variants 
  can also be seen, for example, in \cite{Atkinson1976,Chen,Jerri,Kress,Zemyan}.
 

In this note, we investigate 
 \eqref{Fredholm} 
 for 
  semi-continuous solutions.
 Unlike all existing methods, our strategy here is to make use of Tarski's fixed point theorem. After discussing the  preliminaries, 
 we give sufficient conditions for the existence of monotonic semi-continuous solutions for \eqref{Fredholm}
 using this theorem. 
 Our study not only provides some novel results on the solutions of \eqref{Fredholm}, but also applies a fixed point theorem for lattices, demonstrating a subtle interplay between lattice theory and integral equation theory.
 


In the remainder of the Introduction, we record certain notions from lattice theory that will be used to prove our results in the next section.
As defined in \cite{Szasz}, a relation $\preceq$ on a nonempty set $L$ is
called a {\it partial order} if it is
reflexive (i.e., $x\preceq x$ for all $x\in L$),
antisymmetric (i.e., $x=y$ whenever $x\preceq y$ and $y\preceq x$ in $L$),
and
transitive (i.e., $x\preceq z$ whenever  $x\preceq y$ and $y\preceq z$ in $L$).
$L$ equipped with a partial order $\preceq$
is called a {\it partially ordered set} (or simply a {\it poset}).
For a subset $E$ of the poset $L$,
$b\in L$ is called an {\it upper bound} (resp. a {\it lower bound}) of $E$
if $x\preceq b$ (resp. $b\preceq x$) for all $x\in E$.
Further, $b$ is called the
{\it least upper bound} or {\it supremum} (resp. {\it greatest lower bound} or {\it infimum}),
denoted by $\sup_L E$ (resp. $\inf_L E$),
if $b$ is an  upper bound (resp.  lower bound) of $E$ and
every upper bound (resp. lower bound) $z$ of $E$ satisfies
$b\preceq z$ (resp. $z\preceq b$).
A poset $L$ is said to be {\bf (i)} {\it simply ordered} (or a {\it chain}) if at least one of the relations $x\preceq y$ and $y\preceq x$ hold whenever $x,y \in L$; {\bf (ii)} a {\it lattice} if $\sup_L\{x,y\}$, $\inf_L\{x,y\}\in L$ for every $x,y \in L$; {\bf (iii)} a {\it complete lattice} if $\sup_L E, \inf_L E \in L$ for every nonempty subset $E$ of $L$.
A nonempty
subset $E$ of a lattice $L$ is said to be {\bf (i)} a {\it sublattice} of $L$ if 
$\sup_L\{x,y\}$, $\inf_L\{x,y\}\in E$ for every $x,y \in E$; 
{\bf (ii)} a {\it complete sublattice} of $L$ if 
$\sup_L Y$ and $\inf_L Y$ exist, and both are in $E$ for every nonempty subset $Y$ of $E$.
For convenience, we use  $(L \preceq)$ to denote a lattice $L$ in the partial order $\preceq$.

A map $\phi:L\to L'$, where $(L, \preceq)$ and $(L', \preceq')$ are lattices, is said to be {\it order-preserving} (resp. {\it order-reversing}) if $\phi(x)\preceq' \phi(y)$ (resp. $\phi(y)\preceq' \phi(x)$) whenever  $x\preceq y$ in $L$. A map $K:L\times L'\to L''$ of two variables $x$ and $y$, where $(L'', \preceq'')$ is also a lattice, is said to be {\bf (i)} {\it order-preserving in the variable $x$} (resp. $y$) if $L\ni x\mapsto K(x,y)\in L''$ (resp. $L'\ni y\mapsto K(x,y)\in L''$) is order-preserving for all $y\in L'$ (resp. $x\in L$); {\bf (ii)} {\it order-reversing in the variable $x$} (resp. $y$) if $L\ni x\mapsto K(x,y)\in L''$ (resp. $L'\ni y\mapsto K(x,y)\in L''$) is order-reversing for all $y\in L'$ (resp. $x\in L$).
Let $\mathcal{F}(L,L')$  and $\mathcal{F}_{op}(L,L')$ (resp. $\mathcal{F}_{or}(L,L')$) denote the poset of all maps  and order-preserving (resp. order-reversing)  maps of $L$ into $L'$
respectively in the {\it pointwise partial order} $\trianglelefteq$ defined by $\phi\trianglelefteq \psi$ if $\phi(x)\preceq' \psi(x)$ for all $x\in L$.
Then $\mathcal{F}(L,L')$, $\mathcal{F}_{op}(L,L')$ and $\mathcal{F}_{or}(L,L')$  are 
lattices in the partial order $\trianglelefteq$.

%
%
%
%
%

\begin{lm}{\rm (Tarski \cite{Tarski})}\label{L0}
	If $(L, \preceq)$ is a complete lattice and $f\in \mathcal{F}_{op}(L,L)$, then the set of all fixed points of  $f$ is a non-empty complete sublattice of $L$.
	Further,
	$f$ has the minimum fixed point $x_*$ and the maximum fixed point $x^*$ in $L$ given by
	$x_*=\inf\{x\in L: f(x)\preceq x\}$ and
	$x^*=\sup\{x\in L: x\preceq f(x)\}$.
\end{lm}

The first part of this lemma can also be found in the expository article \cite{subra2000}.
A part of the second, proving that $\sup\{x\in L: x\preceq f(x)\}$ and $\inf\{x\in L: f(x)\preceq x\}$ are fixed points of $f$ thereby proving the existence of a fixed point, can also be found in the book \cite{Gratzer1978}.


\section{Main results}

Let 
$I$ denotes the compact interval $[a,b]$ in $\mathbb{R}$, which is also a simply ordered complete 
lattice in the usual order $\le$.
 As in \cite{royden1988}, a map $\phi:I \to \mathbb{R}$ is said to be {\it USC} (abbreviation for {\it upper-semi-continuous}) (resp. LSC (abbreviation for {\it lower-semi-continuous}))
at $t_0 \in I$ 
if
$\limsup_{t \to t_0} \phi(t) \le \phi(t_0)$ (resp. $\phi(t_0)\le \liminf_{t\to t_0}\phi(t)$), i.e., for every $\epsilon>0$ there exists a $\delta>0$ such that $\phi(t)\le \phi(t_0)+\epsilon$ (resp. $\phi(t_0)-\epsilon \le \phi(t)$) whenever $t\in (t_0-\delta, t_0+\delta)\cap I$.  $\phi$ is said to be USC (resp. LSC) on $I$ if it is USC (resp. LSC) at  each point of $I$. Let
$\mathcal{F}^{usc}_{op}(I,\mathbb{R}_+)$ (resp. $\mathcal{F}^{lsc}_{op}(I,\mathbb{R}_+)$) denotes the set of all USC (resp. LSC) maps in $\mathcal{F}_{op}(I,\mathbb{R}_+)$, where $\mathbb{R}_+:=[0,\infty)$, the set of non-negative reals, which is also a lattice in the usual order $\le$.

\begin{lm}\label{L1}
	$\mathcal{F}^{usc}_{op}(I,\mathbb{R}_+;\kappa):=\{\phi \in \mathcal{F}^{usc}_{op}(I,\mathbb{R}_+):\phi(t)\le \kappa~\text{for all}~t\in I\}$ is a complete lattice in the partial order $\trianglelefteq$ for each $\kappa\ge 0$. This is also true when upper-semi-continuity is replaced by lower-semi-continuity.
\end{lm}

\begin{proof}
Consider an arbitrary subset $\mathcal{E}$ of $\mathcal{F}^{usc}_{op}(I, \mathbb{R}_+;\kappa)$. If $\mathcal{E}=\emptyset$ (resp. $\mathcal{E}\ne\emptyset$), then the map $\Phi:I\to \mathbb{R}_+$ defined by $\Phi(t)=\kappa$ (resp. $\Phi(t)=\inf\{f(t): f \in \mathcal{E}\}$)
	is the infimum of $\mathcal{E}$ in $\mathcal{F}^{usc}_{op}(I,\mathbb{R}_+;\kappa)$.  
	Therefore by Lemma $14$ of \cite{Gratzer1978},
	which says that if every subset of a poset $P$ has the infimum in $P$ then $P$ is complete,
	we see that $\mathcal{F}^{usc}_{op}(I,\mathbb{R}_+;\kappa)$ is a complete lattice. This proves the first part. The proof for the second part is similar. 
\end{proof}

Having the above lemma, we are ready to give our intended results on  
semi-continuous  solutions of \eqref{Fredholm}. 

\begin{theorem}\label{Thm1}
	Let $\lambda,\kappa, \mu, \rho \ge 0$, $f\in \mathcal{F}^{usc}_{op}(I,\mathbb{R}_+;\mu)$, and 
	 $K$ be a continuous map of $I^2:=I\times I$ into $\mathbb{R}_+$
	 order-preserving in both the variables $t$ and $s$ such that  $K(t,s)\le \rho$ for all $(t,s)\in I^2$.
	  Then
	the set $\mathcal{S}_{op}^{usc}(I,\mathbb{R}_+;\kappa)$ of all solutions of \eqref{Fredholm}
	 in $\mathcal{F}_{op}^{usc}(I, \mathbb{R}_+;\kappa)$
	is a non-empty complete sublattice of $\mathcal{F}_{op}^{usc}(I, \mathbb{R}_+;\kappa)$ provided $\kappa(1-\lambda \rho)\ge \mu$.
	Further,
	\eqref{Fredholm} has the minimum solution $\phi_*$ and the maximum solution $\phi^*$ in $\mathcal{F}_{op}^{usc}(I,\mathbb{R}_+;\kappa)$  given by
	\begin{eqnarray*}
		\phi_*&=&\inf\left\{\phi\in \mathcal{F}_{op}^{usc}(I,\mathbb{R}_+;\kappa):   f(t)+\lambda\int_{a}^{b}K(t,s)\phi(s)ds \le \phi(t) ~\text{for all}~t\in I\right\},\\
		\phi^*&=&\sup\left\{\phi\in \mathcal{F}_{op}^{usc}(I,\mathbb{R}_+;\kappa): \phi(t) \le  f(t)+\lambda\int_{a}^{b}K(t,s)\phi(s)ds ~\text{for all}~t\in I\right\}.
	\end{eqnarray*}
\end{theorem}

\begin{proof}
If $\lambda=0$ or $\rho=0$ (resp. $\mu=0$), then  $\phi=f$ is the unique  (resp. $\phi=0$ is trivially a) solution of \eqref{Fredholm} in  $\mathcal{F}_{op}^{usc}(I,\mathbb{R}_+;\kappa)$. So, we may assume that $\lambda, \mu, \rho> 0$. Then $\kappa>0$. 	Define $T$ on $\mathcal{F}_{op}^{usc}(I,\mathbb{R}_+;\kappa)$ by
	\begin{eqnarray}\label{T}
	T\phi(t) = f(t)+\lambda\int_{a}^{b}K(t,s)\phi(s)ds, \quad \forall t\in I.
	\end{eqnarray}

\noindent {\bf Claim:} $T$ is an order-preserving self-map on $\mathcal{F}_{op}^{usc}(I, \mathbb{R}_+;\kappa)$.

	Since $K$ is  order-preserving in the variable $s$, clearly  $s\mapsto K(t,s)\phi(s)$ is an order-preserving map of $I$ into $\mathbb{R}_+$ for each $t\in I$ and $\phi \in \mathcal{F}_{op}^{usc}(I,\mathbb{R}_+;\kappa)$. Therefore $T$ is a well-defined map.

Consider an arbitrary $\phi \in \mathcal{F}_{op}^{usc}(I, \mathbb{R}_+;\kappa)$.
Since $\lambda,\kappa, \mu, \rho > 0$ and $\mu \le \kappa(1-\lambda \rho)$, clearly  $0\le T\phi(t)\le \mu+\lambda\kappa\rho\le \kappa$ for all $t\in I$.  Also, since $f$ is order-preserving on $I$ and $K$ is order-preserving in the variable $t$, it follows that $T\phi$ is order-preserving on $I$.

Next, to prove that $T\phi$ is USC, consider arbitrary $t_0\in I$ and $\epsilon>0$. 
Since $K$ is uniformly continuous on $I^2$, there exists $\delta_1>0$ such that $K(t,s)\le K(t_0,s)+\frac{\epsilon}{2\lambda\kappa(b-a)}$ for all $s\in I$ and $t\in (t_0-\delta_1, t_0+\delta_1)\cap I$.
Since $f$ is USC at $t_0$, there exists $\delta_2>0$ such that $f(t)\le f(t_0)+\frac{\epsilon}{2}$ for all $t\in (t_0-\delta_2, t_0+\delta_2)\cap I$.
Then
\begin{eqnarray*}
T\phi(t)&\le & \left(f(t_0)+\frac{\epsilon}{2}\right)+\lambda \int_{a}^{b}\left(K(t_0,s)+\frac{\epsilon}{2\lambda\kappa(b-a)}\right)\phi(s)ds\\
&\le&f(t_0)+\lambda \int_{a}^{b}K(t_0,s)\phi(s)ds+\epsilon\\
&=&T\phi(t_0)+\epsilon
\end{eqnarray*}
for all $t\in (t_0-\delta, t_0+\delta)\cap I$, where $\delta:=\min\{\delta_1,\delta_2\}$, implying that $T\phi$ is USC at $t_0$.
Further, it is easy to see that $T\phi(t)\le T\psi(t)$ for all $t\in I$ whenever $\phi \trianglelefteq \psi$ in $\mathcal{F}_{op}^{usc}(I, \mathbb{R}_+;\kappa)$. 


Thus, the claim holds, which implies by Lemmas \ref{L0} and the first part of Lemma \ref{L1} that the set of all fixed points of $T$ and hence the set of all solutions of \eqref{Fredholm} in $\mathcal{F}_{op}^{usc}(I, \mathbb{R}_+;\kappa)$ is a non-empty complete sublattice of $\mathcal{F}_{op}^{usc}(I, \mathbb{R}_+;\kappa)$. 
In particular, \eqref{Fredholm} has the minimum solution $\phi_*$ and the maximum solution $\phi^*$ in $\mathcal{F}_{op}^{usc}(I, \mathbb{R}_+;\kappa)$, which are in fact $\min \mathcal{S}^{usc}_{op}(I, \mathbb{R}_+;\kappa)$ and $\max \mathcal{S}^{usc}_{op}(I,\mathbb{R}_+;\kappa)$, respectively.
Further, by Lemma \ref{L0}, we have $\phi_*=\inf\{\phi\in \mathcal{F}_{op}^{usc}(I, \mathbb{R}_+;\kappa):  T\phi \trianglelefteq  \phi\}$ and $\phi^*=\sup\{\phi\in \mathcal{F}_{op}^{usc}(I, \mathbb{R}_+;\kappa): \phi\trianglelefteq  T\phi\}$.
This completes the proof.
\end{proof}

Remark that the current approach with the map $T$ defined in \eqref{T}, employed in the above theorem, can also be used to give a result on order-reversing USC solutions for \eqref{Fredholm}.
More concretely, we have the following result, whose proof is similar to that of the above theorem. 

\begin{theorem}\label{Thm2}
	Let $\lambda,\kappa, \mu, \rho \ge 0$, $f\in \mathcal{F}^{usc}_{or}(I,\mathbb{R}_+;\mu)$, and 
$K$ be a continuous map of $I^2$ into $\mathbb{R}_+$
order-reversing in both the variables $t$ and $s$ such that  $K(t,s)\le \rho$ for all $(t,s)\in I^2$.
Then
the set $\mathcal{S}_{or}^{usc}(I,\mathbb{R}_+;\kappa)$ of all solutions of \eqref{Fredholm}
in $\mathcal{F}_{or}^{usc}(I, \mathbb{R}_+;\kappa)$
is a non-empty complete sublattice of $\mathcal{F}_{or}^{usc}(I, \mathbb{R}_+;\kappa)$ provided $\kappa(1-\lambda \rho)\ge \mu$.
Further,
\eqref{Fredholm} has the minimum solution $\phi_*$ and the maximum solution $\phi^*$ in $\mathcal{F}_{or}^{usc}(I,\mathbb{R}_+;\kappa)$  given by
\begin{eqnarray*}
	\phi_*&=&\inf\left\{\phi\in \mathcal{F}_{or}^{usc}(I,\mathbb{R}_+;\kappa):   f(t)+\lambda\int_{a}^{b}K(t,s)\phi(s)ds \le \phi(t) ~\text{for all}~t\in I\right\},\\
	\phi^*&=&\sup\left\{\phi\in \mathcal{F}_{or}^{usc}(I,\mathbb{R}_+;\kappa): \phi(t) \le  f(t)+\lambda\int_{a}^{b}K(t,s)\phi(s)ds ~\text{for all}~t\in I\right\}.
\end{eqnarray*}
\end{theorem}

It is also worth noting that the above theorems hold when upper-semi-continuity is replaced by lower-semi-continuity. The major difference in their proofs is in the part of establishing the semicontinuity of $T\phi$, which we will discuss in more detail below.

Let the hypotheses of Theorem 1 be satisfied, where lower-semi-continuity is considered instead of upper-semi-continuity, and $\phi \in \mathcal{F}_{op}^{lsc}(I, \mathbb{R}_+;\kappa)$. Set $\alpha= \int_{a}^{b}\phi(s)ds$. Then, as $\phi$ is a non-negative map on $I$, clearly $\alpha\ge 0$. If $\alpha=0$, then $\phi=0$, and therefore $T\phi=f$ is LSC on $I$. So, let  $\alpha>0$.  Also, let $t_0 \in I$ and $\epsilon>0$ be arbitrary. 
Since $K$ is uniformly continuous on $I^2$, there exists $\delta_1>0$ such that $K(t_0,s)-\frac{\epsilon}{2\alpha\lambda}\le K(t,s)$ for all $s\in I$ and $t\in (t_0-\delta_1, t_0+\delta_1)\cap I$. Since $f$ is LSC at $t_0$, there exists $\delta_2>0$ such that $f(t_0)-\frac{\epsilon}{2}\le f(t)$ for all $t\in (t_0-\delta_2, t_0+\delta_2)\cap I$.
Then
\begin{eqnarray*}
	T\phi(t)&\ge & \left(f(t_0)-\frac{\epsilon}{2}\right)+\lambda \int_{a}^{b}\left(K(t_0,s)-\frac{\epsilon}{2\alpha\lambda}\right)\phi(s)ds\\
	&=&f(t_0)+\lambda \int_{a}^{b}K(t_0,s)\phi(s)ds-\epsilon\\
	&=&T\phi(t_0)-\epsilon
\end{eqnarray*}
for all $t\in (t_0-\delta, t_0+\delta)\cap I$, where $\delta:=\min\{\delta_1,\delta_2\}$, implying that $T\phi$ is LSC at $t_0$. Therefore $T\phi$ is LSC on $I$. The proof in the case of Theorem \ref{Thm2} with lower-semi-continuity is similar.

Let $\mathbb{R}_-:=(-\infty,0]$, the set of non-positive reals, which is also a simply ordered lattice in the usual order $\le$.
Although the  solutions of \eqref{Fredholm} in the above discussions are given for $\lambda\ge 0$, we may use them, together with an idea of conjugation through the reflection map, to obtain solutions of it when $\lambda<0$. 
More precisely, the following result shows how to deduce solutions of \eqref{Fredholm}  for $\lambda\le 0$ from those given above for $\lambda\ge 0$ and vice versa.
\begin{lm}
		Let $\lambda,\kappa, \mu, \rho \ge 0$, $f\in \mathcal{F}^{usc}_{op}(I,\mathbb{R}_+;\mu)$, and 
	$K$ be a continuous map of $I^2$ into $\mathbb{R}_+$
	order-preserving in both the variables $t$ and $s$ such that  $K(t,s)\le \rho$ for all $(t,s)\in I^2$. Then $\phi$ is a solution of \eqref{Fredholm} in $\mathcal{F}^{usc}_{op}(I,\mathbb{R}_+;\kappa)$
	if and only if $\tilde{\phi}$ is a solution of 
	\begin{eqnarray*}\label{Fredholm-tilde}
		\tilde{\phi}(t)=\tilde{f}(t)+\tilde{\lambda}\int_{a}^{b}\tilde{K}(t,s)\tilde{\phi}(s)ds, \quad t\in I,
	\end{eqnarray*}
	in $\tilde{\mathcal{F}}^{lsc}_{or}(I,\mathbb{R}_-;\tilde{\kappa}):=\{\psi \in \mathcal{F}^{lsc}_{or}(I,\mathbb{R}_-):\tilde{\kappa}\le \psi(t)~\text{for all}~t\in I\}$, where $\tilde{\lambda}=-\lambda$, $\tilde{\kappa}=-\kappa$, $\tilde{f}=-f\in \tilde{\mathcal{F}}^{lsc}_{or}(I,\mathbb{R}_-;\tilde{\mu})$ with  $\tilde{\mu}=-\mu$, and $\tilde{K}=-K$ is continuous on $I^2$ order-reversing in both of its variables $t$ and $s$ such that $\tilde{\rho}\le \tilde{K}(t,s)$ for all $(t,s)\in I^2$ with $\tilde{\rho}=-\rho$.  Furthermore, 
	this is also true when {\bf (i)} the roles of upper-semi-continuity and lower-semi-continuity are switched; {\bf (ii)}
	 the  roles of order-preserving and order-reversing are switched. 
	
%
\end{lm}
\begin{proof}
It is simple and thus omitted.
\end{proof}

\begin{exmp}
{\rm	Let
	\begin{eqnarray*}
	f(t)=\left\{\begin{array}{cll}
	\frac{t^4}{6}&\text{if}&0\le x<\frac{1}{2},\\
	\frac{t^3}{5}&\text{if}&\frac{1}{2}\le x\le 1,
	\end{array}\right.\quad K(t,s)=4t^7\sin^9\left(\frac{\pi}{2}s\right)~\text{for all}~(t,s)\in I^2,	\end{eqnarray*}
 and $\lambda=\frac{1}{5}$, where $I:=[0,1]$. Then   Theorem \ref{Thm1} is applicable with $\mu=\frac{1}{5}$, $\rho=4$ and $\kappa=2$. }
\end{exmp}

\noindent {\bf Acknowledgment:}
The author is supported by Indian Statistical Institute,
Bangalore in the form of a Visiting Scientist position through the J. C. Bose Fellowship
of Prof. B. V. Rajarama Bhat. 



%
%
%
%
%
%
%


\end{document}